\documentclass[12pt]{amsart}
\numberwithin{equation}{section}
\usepackage{amsmath,amsthm,amsfonts,amscd,eucal}

\usepackage{xy}
\usepackage{graphicx}

\usepackage{amssymb}

\swapnumbers
\theoremstyle{definition}
\newtheorem{thm}{Theorem}[section]
\newtheorem{defini}[thm]{Definition}
\newtheorem{prop}[thm]{Proposition}

\newtheorem{lm}[thm]{Lemma}
\newtheorem{remark}[thm]{Remark}

\newtheorem{notat}[thm]{Notation}

\def\rnum#1{\expandafter{\romannumeral #1}}
\def\Rnum#1{\uppercase\expandafter{\romannumeral #1}}

\newcommand{\diam}{\mathop{\mathrm{diam}}\nolimits}
\newcommand{\dist}{\mathop{\mathrm{dist}}\nolimits}
\newcommand{\id}{\mathop{\mathrm{id}}\nolimits}

\begin{document}


\def\R{{\car}}
\def\K{{\ck}}
\def\G{{\cg}}
\def\Re{{\br}}
\def\Na{{\bn}}
\def\Ze{{\bz}}
\def\Qu{{\bq}}
\def\Co{{\bc}}
\def\H{{\ch}}
\def\<{\leq_\Re}
\def\iff{{\it\thinspace iff\thinspace\ }}
\def\ad{{\rm ad}\thinspace}

\def\D{\Delta}\def\d{\delta}\def\l{\lambda}\def\O{\Omega}\def\o{\omega}\def\f{\phi}
\def\s{\sigma}
\def\t{\theta}
\def\a{\alpha}
\def\p{\psi}\def\m{\mu}\def\n{\nu}
\def\L{\Lambda}\def\e{\varepsilon}\def\r{\rho}\def\b{\beta}
\def\g{\gamma}

\def\ca{{\mathcal A}}
\def\cb{{\mathcal B}}
\def\cc{{\mathcal C}}
\def\cd{{\mathcal D}}
\def\ce{{\mathcal E}}
\def\cf{{\mathcal F}}
\def\cg{{\mathcal G}}
\def\ch{{\mathcal H}}
\def\cai{{\mathcal I}}
\def\cj{{\mathcal J}}
\def\ck{{\mathcal K}}
\def\cl{{\mathcal L}}
\def\cam{{\mathcal M}}
\def\cn{{\mathcal N}}
\def\co{{\mathcal O}}
\def\cp{{\mathcal P}}
\def\cq{{\mathcal Q}}
\def\car{{\mathcal R}}
\def\cs{{\mathcal S}}
\def\ct{{\mathcal T}}
\def\cu{{\mathcal U}}
\def\cv{{\mathcal V}}
\def\cw{{\mathcal W}}
\def\cx{{\mathcal X}}
\def\cy{{\mathcal Y}}
\def\cz{{\mathcal Z}}
\def\alg{{\mathfrak g}}

\def\diam{{\textrm{diam}}}
\def\vol{{\textrm{vol}}}
\def\af{{\textrm{Af}}}
\def\aff{{\textrm{Aff}}}
\def\conv{{\textrm{co}}}

\def\ba{{\mathbb A}}
\def\bc{{\mathbb C}}
\def\bd{{\mathbb D}}
\def\bg{{\mathbb G}}
\def\bi{{\mathbb I}}
\def\bj{{\mathbb J}}
\def\bn{{\mathbb N}}
\def\bq{{\mathbb Q}}
\def\br{{\mathbb R}}
\def\bz{{\mathbb Z}}
\def\bt{{\mathbb T}}
\def\fa{{\mathfrak A}}
\def\fw{{\mathfrak W}}

\title{A Gromov-Hausdorff Distance between\\
von Neumann Algebras and an \\
Application to Free Quantum Fields}

\author{D.\ Guido}
\author{N.\ Marotta}
\author{G.\ Morsella}
\author{L.\ Suriano}
\address{Dipartimento di Matematica, Universit\`a Tor Vergata, Roma, Italy}
\email{guido@mat.uniroma2.it, nunzia.math@gmail.com, morsella@mat.uniroma2.it, lu.suriano@gmail.com}

\date{}

\dedicatory{Dedicated to Sergio Doplicher, on the occasion of his $75^{th}$ birthday.}

\begin{abstract}
A distance between von Neumann algebras is introduced, depending on a further norm inducing the $w^*$-topology on bounded sets. Such notion is related both with the Gromov-Hausdorff distance for quantum metric spaces of Rieffel \cite{R4} and with the Effros-Mar\'echal topology \cite{E1,M} on the  von Neumann algebras  acting on a  Hilbert space.
This construction is tested on the local algebras of free quantum fields endowed with norms related with the Buchholz-Wichmann nuclearity condition \cite{BuWi}, showing the continuity of such algebras w.r.t. the mass parameter.
\end{abstract}

\maketitle

\pagenumbering{arabic}

 \setcounter{section}{-1}

\section{Introduction}
In this note we introduce a suitable notion of distance between von Neumann algebras endowed with a further norm inducing the $w^*$-topology on bounded sets, and apply this construction to the local algebras of the free massive quantum field, showing their convergence to the local algebras of the massless free quantum field.

On the one hand, the mentioned notion of distance between von Neumann algebras is a sort of dual version of the quantum Gromov-Hausdorff distance of Rieffel \cite{R4}.
On the other hand, it is clearly related to the Effros-Marechal topology (\cite{E1,M,HW1,HW2}) on the set of von Neumann algebras acting on a given Hilbert space $\ch$.

Let us recall that, according to Rieffel \cite{R2}, a quantum metric space is a $C^*$-algebra, or more generally an order unit space, endowed with a \emph{Lip-seminorm}\footnote{The original terminology of~\cite{R2} is actually \emph{Lip-norm}, but we will reserve it for a slightly different object, see Def.~\ref{def:Lip-spaces} below. }, namely a densely defined seminorm vanishing only on multiples of the identity which induces a distance on the state space compatible with the $w^*$-topology or, equivalently,
whose closed balls modulo scalars are totally bounded w.r.t. the C$^*$-norm
~\cite{R1, R2, R3a}. Alternatively \cite{GI}, one may assign a densely defined norm $L$ on the $C^*$-algebra $\ca$ whose closed balls are norm compact, so that the dual space $\ca^*$ is endowed with a dual norm $L'$ inducing the $w^*$-topology on bounded sets. 

In our setting, we call Lip-von Neumann algebra (LvNA) a von Neumann algebra $M$ endowed with a dual-Lip-norm $L'$, namely a norm inducing the $w^*$-topology on bounded sets. In this way the predual $M_*$ is endowed with a Lip-norm $L$, which is  a densely defined norm whose closed balls are norm compact.
We note that $L'$ gives rise to a Hausdorff distance between $w^*$-compact subsets of $M$. 
Then, following the ideas of Rieffel~\cite{R4}, one may proceed to define a Gromov-Hausdorff pseudo-distance between pairs $(M_1,L'_1)$, $(M_2,L'_2)$ as a distance between the corresponding unit balls, embedded in the direct-sum von Neumann algebra, but, as in his case, distance zero does not imply isomorphic von Neumann algebras. Therefore, following \cite{Kerr, KL}, we replace unit balls with the positive part of the unit balls of $\cam_2(M_i)$, thus getting a distance between isomorphism classes of Lip-von Neumann algebras (cf. Definition \ref{def:dqGH}). 

Let us mention here that, besides the method proposed by Kerr and Li, other proposals have been made in order to get a notion of Gromov-Hausdorff distance more tailored for C$^*$-algebras, cf. e.g. \cite{Li,L1,L2}.

Recall~\cite{E1,M,HW1,HW2} also that, for a separable Hilbert space $\ch$, the Effros-Marechal topology on the set of von Neumann algebras acting on a given Hilbert space $\ch$ may be metrized as follows: one chooses a distance on $\cb(\ch)$ which metrizes the $w^*$-topology on bounded subsets and  considers the corresponding Hausdorff  distance on $w^*$-compact sets. Then  the distance between two von Neumann algebras on $\ch$ may be defined as the Hausdorff distance between their unit balls. In this sense our construction is a local version of the Effros-Marechal topology, since it metrizes the $w^*$-compact sets of a given von Neumann algebra, instead of all the $w^*$-compact sets in $\cb(\ch)$, and our distance between isomorphism classes of LvNA is at the same time a Gromov-Hausdorff version of the Effros-Marechal topology.

In algebraic quantum field theory, namely the description of relativistic quantum physics by means of operator algebras, various notions of compactness or, more strongly, nuclearity properties have been considered. We focus here on the Buchholz-Wichmann nuclearity condition \cite{BuWi}. It provides natural, physically meaningful dual-Lip-norms for the algebras of local observables, which can be used to study various notions of distance or convergence. In this paper we use them to prove that the local algebras of the free quantum field are continuous in the mass parameter and, in particular, that, for any bounded region $\co$, the algebras $\ca_m(\co)$ converge to $\ca_0(\co)$ when $m$ goes to zero.

We conclude this introduction mentioning that the idea of some kind of topological convergence for the algebras of local observables in quantum field theory originally came from early discussions of the first named author with R. Verch and concerned the scaling limit~\cite{BV1}, however with the expectation that Rieffel's quantum Gromov-Hausdorff distance between $C^*$-algebras could be used. Only later, during discussions of the first named author with H. Bostelmann, the idea rose that the notion had to be dualized and adapted to von Neumann algebras. The related technical work is contained in the PhD thesis of the last named author, under the supervision of the first named author. Finally, in the master thesis of the second named author, under the supervision of the first and third author, the convergence of the free local algebras w.r.t. the mass parameter was studied.

The content of this paper was reported by the third author in the Marcel Grossman meeting in Rome, July 2015.

\section{Lip-von Neumann Algebras}
\label{ch:LVNA}

\pagestyle{plain}

In this section, we will introduce the notion of Lip-von Neumann Algebra\footnote{In what follows, we will consider concrete von Neumann algebras, but the same definitions may refer also to abstract $W^*$-algebras, as long as we do not make any reference to the representing Hilbert space.}.
The reason why we use the term ``Lip'' (which stands for ``Lipschitz'') will be clear in a while. The leading idea is in fact to ``dualize'' the notion of quantum metric spaces given by Rieffel in \cite{R4}. One may reformulate  Rieffel's approach for $C^*$-algebras as follows: assign a further norm $L$ on the $C^*$-algebra $\ca$ in such a way that the dual norm $L'$ metrizes the $w^*$-topology on bounded subsets of $\ca^*$.
If instead one starts with a von Neumann algebra, which is canonically a dual space, it becomes quite natural to assign a norm $L$ on its pre-dual $M_*$ in such a way that the corresponding dual norm $L'$ on $M$ metrizes the $w^*$-topology on bounded subsets of $M$.



Before translating these ideas into a definition, we briefly recall some basic facts about the notion of Lip-space introduced in \cite{GI}.

\begin{defini}\label{def:Lip-spaces}
{\em We call} Lip-space (LS) {\em a triple $(X, \| \cdot \|, L)$, where:}
\begin{itemize}
\item[(i)] {\em $(X,\|\cdot\|)$ is a Banach space,}
\item[(ii)] {\em $L: X \to [0, +\infty]$ is finite on a dense vector subspace $\cl$, where it is a norm,}
\item[(iii)] {\em the unit ball w.r.t. $L$, namely $\cl_1 := \{ x \in X: L(x) \leq 1\}$, is compact w.r.t. the Banach norm of $X$.}
\end{itemize}
{\em We call \emph{Lip-norm} a norm $L$ satisfying properties \emph{(ii)} and \emph{(iii)} above.} 
\end{defini}

We then call {\it radius} of the Lip-space $(X, \| \cdot \|, L)$, and denote it by $R$, the maximum value of the (Banach) norm on $\cl_1$, namely
\begin{equation}
R := \max \{ \| x \|: x \in \cl_1 \}.
\end{equation}
It then follows that
\begin{equation}
\| x \| \leq R L(x), \quad \forall x \in X.
\end{equation}

For the reader convenience, we also recall Prop.~2.3 of \cite{GI}.

\begin{prop}\label{prop:GI}
{\em Let $(X, \| \cdot \|, L)$ be a Lip-space. Then, the norm $L'$ on the Banach dual $X^*$ given by
\begin{equation}\label{DualLipNorm}
L'(\xi) := \sup\left\{| \langle \xi, x \rangle |\,:\,x \in X,\,L(x) \leq 1 \right\}
\end{equation}
induces the $w^*$-topology on the bounded subsets of $X^*$, and the radius $R$ is also equal to the radius of the unit ball of $(X^*, \| \cdot \|)$ w.r.t. the $L'$-norm.}
\end{prop}

Our aim is to give an intrinsic characterization of the norm $L'$ on $X^*$. Given a seminorm $L$ on a subspace $\cl$ of a Banach space $X$, we can of course extend it to $X$ by setting it equal to $\infty$ on $X \setminus \cl$. For simplicity, we will still refer to such an extension as a seminorm on $X$. Moreover, if $L'$ is a seminorm (in this sense) on $X^*$, we can associate to it a seminorm $L$ on $X$ by a formula which is ``dual'' to~\eqref{DualLipNorm}:
\begin{equation}\label{PredualLipNorm}
L(x) := \sup\{| \langle \xi, x \rangle |\,:\,\xi \in X^*,\,L'(\xi) \leq 1\}.
\end{equation}

\begin{lm}\label{lm:seminormduality}
{\em Let $X$ be a Banach space, $X^*$ its Banach dual, $L$ a seminorm on $X$ and $L'$ a seminorm on $X^*$.  Then $C := \{x \in X\,:\,L(x) \leq 1\}$ is norm closed and~\eqref{DualLipNorm} is satisfied if and only if $Q := \{\xi \in X^*\,:\,L'(\xi) \leq 1\}$ is norm closed and~\eqref{PredualLipNorm} is satisfied.}
\end{lm}

\begin{proof}
The set $C$ is convex and balanced, and if~\eqref{DualLipNorm} holds its polar set $C^\circ$ satisfies
\begin{align*}
C^\circ&=\{\xi\in X^*:|\langle\xi,x\rangle|\leq1,L(x)\leq1\}\\
&=\{\xi\in X^*:L'(\xi)\leq1\} = Q.
\end{align*}
Therefore $Q$, as a polar set, is norm closed, and if $C$ is norm closed too, by the Bipolar Theorem~\cite{Ed} $Q^\circ = C^{\circ\circ} = C$, which entails~\eqref{PredualLipNorm}. The proof of the converse statement is completely analogous.
\end{proof}

We can now give the notion of dual Lip-space.

\begin{defini}\label{DLS}
{\em A} dual Lip-space {\em is a Banach space $Y$ endowed with a norm $L'$ which 
induces a compact topology on the closed unit ball. We will call such a norm a dual-Lip-norm.}
\end{defini}


\begin{prop}\label{DLSiff}
{\em 
Let $X$ be a Banach space, $X^*$ its Banach dual.
\begin{itemize}
\item[(i)] Given a Lip-norm $L$ on $X$, formula $\eqref{DualLipNorm}$ defines a dual-Lip-norm on $X^*$.
\item[(ii)] Given a dual-Lip-norm $L'$ on $X^*$, formula \eqref{PredualLipNorm}
defines a Lip-norm on $X$.
\item[(iii)] The constructions in (i) and (ii) are inverse each other, namely $L$ on $X$ produces $L'$ on $X^*$ iff $L'$ on $X^*$ produces $L$ on $X$.
\item[(iv)] A dual-Lip-norm $L'$ on $X^*$ induces the $w^*$-topology on the bounded subsets.
\end{itemize}
}
\end{prop}

\begin{proof}
We will use the notations $C := \{x \in X\,:\,L(x) \leq 1\}$, $Q := \{\xi \in X^*\,:\,L'(\xi) \leq 1\}$ as in Lemma~\ref{lm:seminormduality}.

(i) It follows by Proposition \ref{prop:GI}.

(ii) Since $L'$ is a norm (everywhere finite on $X^*$), it is clear that $L(x) = 0$ only if $x= 0$. The normic unit ball $X^*_1:=\{\xi \in X^*: \| \xi \| \leq 1\}$ is $L'$-compact by hypothesis, hence there exists $R>0$ such that $L'(\xi)\leq R\|\xi\|$, which implies  that $Q$ is  norm closed. Therefore by Lemma~\ref{lm:seminormduality}, $C$ is norm closed. We now prove that it is norm compact.
Identifying $X$ with its isometric image in the bidual $X^{\ast \ast}$ of $X$, we may consider the set $C$ as a family of 
functions on the $L'$-compact set $X^*_1 \subseteq R\cdot Q$. Since $| \langle \xi, x \rangle | \leq L(x) L'(\xi)$, we see that $C$ is  uniformly bounded by $R$ on $X^*_1$. Moreover, as $| \langle \xi_1, x \rangle - \langle \xi_2, x \rangle | \leq L'(\xi_1 - \xi_2)$ on $C$, 
then  $C$ is also $L'$-equicontinuous. Therefore, by the Ascoli-Arzel\`a Theorem (see, {\em e.g.}, \cite{Ru}), $C$ is compact in the sup-norm $\| \cdot \|_{\infty}$, which coincides on it with the original (Banach) norm of $X$. It remains to check that $\cl = \{x \in X\,:\,L(x)<\infty\}$ is dense. If not, by the Hahn-Banach Theorem one could find $\xi \in X^*\setminus\{0\}$ vanishing on $\cl$, but since~\eqref{DualLipNorm} holds by Lemma~\ref{lm:seminormduality}, this is incompatible with the norm property of $L'$.

(iii) The result follows by Lemma~\ref{lm:seminormduality}: starting from $(X,L)$, by hypothesis $C$ is norm compact, hence norm closed, therefore~\eqref{PredualLipNorm} holds; starting from $(X^*,L')$,  $Q$ is  norm closed as observed in the proof of (ii), and then one gets~\eqref{DualLipNorm}.

(iv) By points (ii) and (iii), any $L'$ is generated by an $L$ on $X$, therefore the thesis follows by Proposition \ref{prop:GI}.
\end{proof}

We now specialize to the case in which $X^*$ is a von Neumann algebra.

\begin{defini}\label{def:LipIso}
{\em A Lip-von Neumann algebra (LvNA) is a von Neumann algebra $M$ endowed with a further norm $L'$ which 
metrizes the $w^*$-topology on bounded subsets.

The LvNA's $(M, L'_M)$ and $(N, L'_N)$ are said to be} isomorphic (as Lip-von Neumann algebras) {\em if there is an isometric ${}^*$-isomorphism $\varphi: M \to N$, such that}
\begin{equation}
L'_N (\varphi(x)) = L'_M(x), \quad for \; any \quad x \in M.
\end{equation}
\end{defini}

By the previous discussion it turns out that a Lip-von Neumann algebra is a von Neumann algebra $M$ which is also a dual Lip-space, or equivalently, it is a vNA whose predual $M_{\ast}$ is a Lip-space. Also, to assign a dual-Lip-norm on $M$ is equivalent to assign a Lip-norm on its predual $M_{\ast}$. Finally, the request that $L'$ metrizes the $w^*$-topology on bounded subsets can be replaced by the request that $M_1 = \{x \in M\,:\,\|x\|\leq 1\}$ is $L'$-compact.

\section{The Dual Quantum Gromov-Hausdorff Distance ${\rm dist}_{qGH^*}$.}
\label{ch:results}

\pagestyle{plain}

\subsection{The Effros-Mar\'echal Topology}\label{subsec:EMtop}

Let $\ch$ be a (fixed) Hilbert space, and let vN($\ch$) be the set of all von Neumann subalgebras of $\cb(\ch)$. We can endow the space vN($\ch$) with a certain natural topological structure. The definition of this topology goes back to the works of Effros \cite{E1} and Mar\'echal \cite{M}, and has been further studied by Haagerup and Winslow in the papers \cite{HW1,HW2}. The original definition of Mar\'echal is the following:

\begin{defini}\label{def:EMtop}
{\em The} Effros-Mar\'echal topology {\em is the weakest topology on vN($\ch$) in which the maps
$$
\mbox{vN}(\ch) \ni M \mapsto \| \varphi |_M \|
$$
are continuous on vN($\ch$) for every $\varphi \in \cb(\ch)_{\ast}$.}
\end{defini}

We now specialize to the separable case. If $\ch$ is separable, then vN($\ch$) is a Polish space in this topology \cite{M}, {\em i.e.}, the Effros-Mar\'echal topology is metrizable, second countable and complete. In order to construct a metric on vN($\ch$) which induces the Effros-Mar\'echal topology, one takes any distance $\r$ on $B(\ch)$ inducing the \emph{wo}-topology on bounded subsets of $B(\ch)$ (which coincides with the $\s$-weak or $w^*$-topology on bounded sets).
The corresponding Hausdorff distance between unit balls of von Neumann algebras in vN($\ch$) will be then a metric for the Effros-Marechal topology (this is a consequence of \cite{E2}), that is, for $M_1, M_2 \in \mathrm{vN}(\ch)$, $$\dist_{EM} (M_1, M_2) := \dist_H^{\r} ((M_1)_1, (M_2)_1), $$
where we recall that, for $C_1,C_2$ closed subsets of a compact metric space $(X,\rho)$, their Hausdorff distance (see e.g. \cite{Bu}) is given by
\begin{equation}\label{Hmetric}
\dist^\rho_H(C_1,C_2)= \max \Big( \sup_{x \in C_1} \inf_{y \in C_2} \r(x, y) , \sup_{y \in C_2}  \inf_{x \in C_1} \r(x, y) \Big).
\end{equation}

\subsection{The Distance ${\rm dist}_{qGH^*}$}

As seen in the previous paragraph, it is possible to define a Hausdorff-like distance between two von Neumann subalgebras of $\cb(\ch)$ for a separable Hilbert space $\ch$,  provided one chooses a norm on $\cb(\ch)$ inducing the $w^*$-topology on bounded sets. 
As in the case of ordinary compact metric spaces, one may proceed from the Hausdorff distance between closed subsets of a given compact metric space to the Gromov-Hausdorff distance, which is a pseudo-distance between compact metric spaces considered \emph{per se}. This pseudo-distance then becomes a true distance on the space of isometry equivalence classes of compact metric spaces.

Let $M$, $N$ be two Lip-von Neumann algebras with dual-Lip-norms $L'_M$, $L'_N$. We want to introduce a Gromov-Hausdorff-type notion of distance between them. In order to have isomorphic Lip-von Neumann algebras (according to Def.~\ref{def:LipIso}) whenever they are at distance zero 
we follow Kerr~\cite{Kerr} and consider not only the original algebras $M$ and $N$, but also the $2 \times 2$-matrix algebras $\mathcal{M}_2(M)$ and $\mathcal{M}_2(N)$ with entries in $M$ and $N$, respectively. We introduce some notation.

\begin{notat}
From now on, we will write $L$ instead of $L'$ when dealing with dual-Lip-norms on von Neumann algebras. Moreover, given a LvNA $(M,L_M)$, we still denote by $L_{M}$ the dual-Lip-norm on $\mathcal{M}_2(M)$ induced by that on $M$ as follows:
\begin{equation}\label{lip2by2}
L _{M} ((a_{ij}))  := \max_{i, j = 1, 2} ( L_M (a_{ij})), \qquad (a_{ij}) \in \mathcal{M}_2(M).
\end{equation}
Notice that the extended $L_{M}$ gives back the original Lip-norm, when restricted to the copy of $M$ diagonally embedded in $\mathcal{M}_2(M)$.
\end{notat}

\begin{lm}\label{lm:2by2}
\emph{If $(M, L_M)$ is a Lip-von Neumann algebra, then $(\mathcal{M}_2(M), L_M)$ is a LvNA, too. Moreover, they have the same radius.}
\end{lm}

\begin{proof}
Indeed, $\mathcal{M}_2(M_1) \cong M_1 \otimes F_2$, where $F_2$ is the type $I_2$ factor, and so, if $L_M$ induces  the $w^*$-topology on the unit ball $M_1$ of $M$, then its ``lift'' to $\mathcal{M}_2(M)$ will induce the same topology on $\mathcal{M}_2(M_1)$. Since $\mathcal{M}_2(M)_1$ is a $w^*$-closed subset of $\mathcal{M}_2(M_1)$, the dual Lip-norm property follows.

As for the radii, as $R_M = \max_{x \in M_1} L_M(x)$, and by $w^*$-compactness the maximum is realized by some element $x_0 \in M_1$, then
$$
R_{\cam_2(M)} = \max_{(x_{i j}) \in \cam_2(M)_1} L_M((x_{i j})) \geq L_M \left( \left(
\begin{array}{cc}
x_0 & 0 \\
0 & 0
\end{array}
\right) \right) = R_M;
$$
on the other hand, as $(x_{i j}) \in \cam_2(M)_1$ implies $x_{i j} \in M_1$, $i, j = 1, 2$, one has
$$
L_M((x_{i j})) = \max_{i j} L_M(x_{i j}) \leq R_M,
$$
hence, $R_{\cam_2(M)} \leq R_M$, and the claim follows. 
\end{proof}

We now introduce our notion of distance between Lip-von Neumann algebras.

\begin{defini}\label{def:dqGH}
\emph{Let $(M, L_M)$ and $(N, L_N)$ be Lip-von Neumann algebras, and denote by $\mathcal{L} ((M,L_M), (N,L_N))$ the set of all the seminorms $L = L_{M \oplus N}$ on the direct sum $M \oplus N$, s.t. $L |_{M \oplus \{ 0 \}} = L_M$ and $L |_{\{ 0 \} \oplus N} = L_N$. We define the} dual quantum Gromov-Hausdorff distance \emph{between $(M,L_M)$ and $(N,L_N)$, by setting}
\begin{equation}
{\rm dist}_{qGH^*}((M,L_M), (N,L_N)) := \inf_{L \in \mathcal{L} ((M,L_M), (N,L_N))} {\rm dist}^{L}_H (X_M, X_N),
\end{equation}
\emph{where $X_M$ (resp., $X_N$) denotes the positive part of the unit ball of $2\times2$ matrices on $M$ (resp., $N$)}.
\end{defini}

To simplify the notation, we will drop the indication of the Lip-norms $L_M$, $L_N$ from the set $\mathcal{L}((M,L_M), (N,L_N))$ and from the distance $\dist_{qGH^*}((M,L_M), (N,L_N))$ whenever no ambiguity can arise. In analogy with Lm.~13.6 in \cite{R4}, we  have the following.

\begin{lm}\label{lm:rad}
{\em Let $(M, L_M)$, $(N, L_N)$ be Lip-von Neumann algebras, and let $R_M$, $R_N$ be the respective radii. Then,}
\begin{equation}
| R_M - R_N | \leq {\rm dist}_{qGH^*} (M, N) \leq R_M + R_N.
\end{equation}
\end{lm}

\begin{proof}
Let us prove the second inequality, from which the finiteness of the distance follows. For $L\in \mathcal{L} (M, N)$ we have
$$
L(x, -y) \leq L_M(x) + L_N(y) \leq R_M + R_N, \quad x \in X_M, \, y \in X_N,
$$
where we used the equality of the radii in Lemma \ref{lm:2by2}.
The second inequality directly follows from the definition of Hausdorff distance \eqref{Hmetric}.

Now, let $d = {\rm dist}_{qGH^*} (M, N)$. Given $\e > 0$, we can find an $L \in \cl(M, N)$ s.t. $\dist^{L}_H (X_M, X_N) < d + \e$. Then, for any $x \in X_M$, there is an $y \in X_N$ such that
$$
L_M(x) \leq L(x, -y) + L_N(y) < d + \e + R_N.
$$
Since $\e$ is arbitrary, it follows that
$$
R_M \leq d + R_N.
$$
Reversing the roles of $M$ and $N$, we obtain also
$$
R_N \leq d + R_M,
$$
and the first inequality is proven.
\end{proof}

We need to show that ${\rm dist}_{qGH^*}$ is a metric. It is clearly symmetric in $M$ and $N$. 

\begin{thm}[\textbf{Triangle Inequality}]\label{thm:TI}
\emph{Let $(M_1, L_1)$, $(M_2, L_2)$, and $(M_3, L_3)$ be Lip-von Neumann algebras. Then}
\begin{equation}
{\rm dist}_{qGH^*}(M_1, M_3) \leq {\rm dist}_{qGH^*}(M_1, M_2) + {\rm dist}_{qGH^*}(M_2, M_3).
\end{equation}
\end{thm}

\begin{proof}
Let $1\geq \e > 0$ be given. Then, we can find an $L_{12} \in \mathcal{L} (M_1, M_2)$ s.t.
$$
{\rm dist}^{L_{12}}_H (X_{M_1}, X_{M_2}) \leq {\rm dist}_{qGH^*}(M_1, M_2) + \e/2.
$$
Similarly, we can find $L_{23} \in \mathcal{L} (M_2, M_3)$ s.t.
$$
{\rm dist}^{L_{23}}_H (X_{M_2}, X_{M_3}) \leq {\rm dist}_{qGH^*}(M_2, M_3) + \e/2.
$$
We define 
$$
L_{13}(x_1, x_3) := \inf_{x_2 \in M_2} (L_{12}(x_1, -x_2) + L_{23} (x_2, x_3)).
$$
We shall prove that it is a seminorm, whose restrictions to $M_1$ and $M_3$ are $L_1$ and $L_3$, respectively. Indeed, the positive homogeneity is clear, and we have
\begin{eqnarray*}
& & L_{13}(x_1 + y_1, x_3 + y_3) \\
&=& \inf_{x_2 \in M_2} (L_{12}(x_1+ y_1, -x_2) + L_{23} (x_2, x_3 + y_3)) \\
&=& \inf_{x_2,  y_2 \in M_2} (L_{12}(x_1+ y_1, -(x_2 + y_2)) + L_{23} (x_2 + y_2, x_3 + y_3)) \\
&\leq& \inf_{x_2 , y_2 \in M_2} (L_{12}(x_1, -x_2) + L_{12}(y_1, -y_2) + L_{23} (x_2, x_3) + L_{23} (y_2, y_3)) \\
&=& \inf_{x_2 \in M_2} (L_{12}(x_1, -x_2) + L_{23} (x_2, x_3)) + \inf_{y_2 \in M_2} (L_{12}(y_1, -y_2) + L_{23} (y_2, y_3)) \\
&=& L_{13}(x_1, x_3) + L_{13}(y_1, y_3).
\end{eqnarray*}
Then, let us check the restriction requirement: since $L_{23} (x_2, 0) =  L_2 (x_2) = L_{12}(0, x_2)$, we have
\begin{eqnarray*}
L_{13}(x_1, 0) &=& \inf_{x_2 \in M_2} (L_{12}(x_1, -x_2) + L_{23} (x_2, 0)) \\
&=& \inf_{x_2 \in M_2} (L_{12}(x_1, -x_2) + L_2 (x_2)) \leq L_{12}(x_1, 0) = L_1 (x_1), \\
\end{eqnarray*}
and, since $L_{12}(x_1, -x_2) = L_{12}((x_1, 0) + (0, -x_2)) \geq | L_1 (x_1) - L_2(x_2) |$,
\begin{equation*}
L_{13}(x_1, 0) \geq \inf_{x_2 \in M_2} (| L_1 (x_1) - L_2(x_2) | + L_2 (x_2)) = L_1 (x_1),
\end{equation*}
and so $L_{13}(x_1, 0) = L_1 (x_1)$. Similarly, one shows that $L_{13}(0, x_3) = L_3 (x_3)$.
\\
Now, suppose that ${\rm dist}^{L_{12}}_H (X_{M_1}, X_{M_2}) = d_{12}$ and ${\rm dist}^{L_{23}}_H (X_{M_2}, X_{M_3}) = d_{23}$. By definition of Hausdorff distance and by $w^*$-compactness of the positive part of the unit ball, for any given $x_1 \in X_{M_1}$ we can find an $x_2 \in X_{M_2}$ - call it $f(x_1)$ - s.t. $L_{12} (x_1, -x_2) = L_{12} (x_1, -f(x_1)) \leq d_{12}$, and, analogously, for any $x_2 \in X_{M_2}$, we can find a corresponding $x_3 \in X_{M_3}$ -- call it $g(x_2)$ -- s.t. $L_{23} (x_2, -x_3) = L_{23} (x_2, -g(x_2)) \leq d_{23}$. In other words, for any given $x_1 \in X_{M_1}$, we can find an $x_3 = g(f(x_1)) \in X_{M_3}$, s.t.
\begin{eqnarray*}
L_{13} (x_1, -g(f(x_1))) &\leq& L_{12} (x_1, -f(x_1)) + L_{23} (f(x_1), -g(f(x_1))) \\ 
&\leq& d_{12} + d_{23}.
\end{eqnarray*}
Similarly, for any given $x_3 \in X_{M_3}$, we can find an $x_1 = h(k(x_3)) \in X_{M_1}$, s.t.
\begin{eqnarray*}
L_{13} (h(k(x_3)), -x_3) &\leq& L_{12} (h(k(x_3)), -k(x_3)) + L_{23} (k(x_3), -x_3) \\
&\leq& d_{12} + d_{23}.
\end{eqnarray*}
Since this holds for any $x_1$ in $X_{M_1}$ and for any $x_3$ in $X_{M_3}$, we obtain
\begin{eqnarray*}
{\rm dist}^{L_{13}}_H (X_{M_1}, X_{M_3}) &\leq& {\rm dist}^{L_{12}}_H (X_{M_1}, X_{M_2}) + {\rm dist}^{L_{23}}_H (X_{M_2}, X_{M_3}) \\
&\leq& {\rm dist}_{qGH^*}(M_1, M_2) + {\rm dist}_{qGH^*}(M_2, M_3) + \e.
\end{eqnarray*}
Therefore, taking the infimum on the l.h.s., we obtain
$$
{\rm dist}_{qGH^*}(M_1, M_3) \leq {\rm dist}_{qGH^*}(M_1, M_2) + {\rm dist}_{qGH^*}(M_2, M_3) + \e,
$$
and so, by the arbitrariness of $\e$, the thesis follows.
\end{proof}

Finally, we want to show that, if two Lip-von Neumann algebras have distance ${\rm dist}_{qGH^*}$ equal to zero, then they are isomorphic as LvNA (see Definition \ref{def:LipIso}, and also Definition 2.4 and Thm.~4.1 in \cite{Kerr}). The proof is inspired by Rieffel's proof of the analogous property for the quantum Gromov-Hausdorff distance between compact quantum metric spaces, completed by the $2\times2$ matrix trick of Kerr \cite{Kerr}.


\begin{lm}\label{lm:equi1}
{\em The family $\mathcal{L} (M, N)$ is uniformly $w^*$-equicontinuous on the unit ball of $M\oplus N$.}
\end{lm}

\begin{proof}
For any $\e > 0$, and for any given $x_0 \in M_1$, $y_0 \in N_1$, let
\begin{eqnarray*}
{\mathcal N}(x_0, \e/2) &=& \{ x \in M : L_M (x - x_0) < \e/2 \}, \\
{\mathcal N}(y_0, \e/2) &=& \{ y \in N : \, L_N (y - y_0) < \e/2 \},
\end{eqnarray*}
so that the subset $\mathcal{N}(x_0, y_0; \e/2) := \{ (x, -y) \in M \oplus N: \, x \in \mathcal{N}(x_0, \e/2), y \in {\mathcal N}(y_0, \e/2) \}$ is a $w^*$-neighborhood of $(x_0, -y_0)$.  If $(x, -y) \in \mathcal{N}(x_0, y_0; \e/2)$, then, for any $L \in \mathcal{L} (M, N)$, we have
\begin{eqnarray*}
| L(x, -y) - L(x_0, -y_0)| & \leq & | L(x, -y) - L(x, -y_0) | + \\
& & \quad | L(x, -y_0) - L(x_0, -y_0) | \\
& \leq & L_M (x - x_0) + L_N (y - y_0) < \e,
\end{eqnarray*}
hence, $\mathcal{L} (M, N)$ is uniformly $w^*$-equicontinuous.
\end{proof}

%
%

\begin{lm}\label{lm:unique1}
{\em Let $L \in \mathcal{L} (M, N)$. For each $x \in \cam_2(M)$ there is at most one $y \in \cam_2(N)$ such that $L (x, -y) = 0$, and similarly for each $y \in \cam_2(N)$.}
\end{lm}

\begin{proof}
If $L (x, -y) = 0 = L (x, -y')$, then
\begin{eqnarray*}
L_N (y - y') &=& L ( (0, -y) - (0, -y')) = L ( (x, -y) - (x, -y')) \\
&\leq& L (x, -y) + L (x, -y') = 0, 
\end{eqnarray*}
and thus $y' = y$.
\end{proof}

Now, we can prove the following

\begin{thm}\label{thm:dzero}
{\em Let $(M, L_M)$ and $(N, L_N)$ be Lip-von Neumann algebras. If}
$$
\dist_{qGH^*} (M, N) = 0,
$$
{\em then $(M, L_M)$ and $(N, L_N)$ are isomorphic Lip-von Neumann algebras.}
\end{thm}

\begin{proof}
We shall proceed by steps.

{\em Claim 1. There exists a seminorm $L_0 \in \cl(M,N)$ and a unital positive bijective map $\varphi: \cam_2(M) \to \cam_2(N)$ with unital positive inverse $\varphi^{-1}$ s.t. $L_0(x, -\varphi(x)) = 0$ for any $x \in \cam_2(M)$.}

In fact, if $\dist_{qGH^*} (M, N) = 0$, then we can find a sequence $\{L_n \}_{n \in \bn}$ of seminorms in $\cl(M, N)$ s.t.
$$
\dist_H^{L_n} (X_M, X_N) < \frac{1}{n}.
$$
Clearly, the sequence $\{ L_n \}$ is uniformly bounded on the unit ball of $M \oplus N$  by $R_M+ R_N$ (see the proof of Lm.~\ref{lm:rad}), where $R_M$ (resp., $R_N$) is the radius of $(M, L_M)$ (resp., $(N, L_N)$). Since it is also $w^*$-equicontinuous by Lm.~\ref{lm:equi1}, we can apply the Ascoli-Arzel\`a Theorem (see, \emph{e.g.}, \cite{Ru}) to conclude that it admits a uniformly convergent subsequence, which we will still denote by $\{L_n\}$ for simplicity. It is then clear that, by homogeneity, this sequence converges uniformly on bounded subsets, and therefore pointwise everywhere on $M\oplus N$, and we denote by $L_0$ its limit. Then obviously $L_0 \in \mathcal{L}(M,N)$, and for all $\varepsilon > 0$ we can find $n_\varepsilon \in \mathbb{N}$ such that for all $n > n_\varepsilon$
\[
|L_n(x) - L_0(x)| < \varepsilon \|x\|, \qquad \forall\, x \in M\oplus N.
\]
We now observe that the liftings of $L_n$ to $\cam_2(M) \oplus \cam_2(N)$ are again converging to the lifting of $L_0$ to $\cam_2(M) \oplus \cam_2(N)$, uniformly on bounded sets. In fact, given $x:=(x_{ij})\in \cam_2(M) \oplus \cam_2(N)=\cam_2(M\oplus N)$, we have 
\begin{align*}
L_n(x)&-L_0(x)
=\max_{ij}L_n(x_{ij})-\max_{ij}L_0(x_{ij})\\
&=L_n(x_{i_1j_1})-L_0(x_{i_2j_2})\\
&=(L_n(x_{i_1j_1})-L_0(x_{i_1j_1}))+(L_0(x_{i_1j_1})-L_0(x_{i_2j_2}))\\
&\leq\e\|x_{i_1j_1}\|\leq\e\|x\|,
\end{align*}
where $i_1,j_1$ realize the maximum for $L_n(x_{ij})$ and $i_2,j_2$ realize the maximum for $L_0(x_{ij})$. The  inequality $L_n(x)-L_0(x)\geq-\e\|x\|$ is obtained analogously.
We now show that $\dist_H^{L_0} (X_M, X_N)=0$. 
In fact, given $\e > 0$, we can find an $n_{\e}$ such that for all $n \geq n_{\e}$, we have $| L_0 (x,  -y) - L_n (x,  -y) | < \e/2$ for all $x \in X_M$, $y \in X_N$, and thus
\begin{eqnarray*}
L_0 (x,  -y) &\leq& | L_0 (x,  -y) - L_n (x,  -y) | + L_n (x,  -y) \\
&\leq& \e/2 + L_n (x,  -y).
\end{eqnarray*}
Hence, keeping $x \in X_M$ fixed, if we take the infimum over all $y \in X_N$, we obtain, for $n$ sufficiently large (with $1/n < \e/2$),
$$
\inf_{y \in X_N} L_0 (x,  -y) \leq \e/2 + \inf_{y \in X_N} L_n (x,  -y) < \e/2 + \e/2 = \e.
$$
By arbitrariness of $\e$ we get $\inf_{y \in X_N} L_0 (x,  -y)=0$ and, by compactness, the  infimum is actually a minimum, namely for any $x\in X_M$ we find a unique $y:=\varphi(x)\in X_N$ such that $L_0 (x,  -y)=0$, where uniqueness follows by Lm.~\ref{lm:unique1}.
Reversing the role of $M$ and $N$ one sees that such map has an inverse defined on $X_N$, namely $\varphi$ is surjective. This proves $\dist_H^{L_0} (X_M, X_N)=0$.

Notice that $\varphi$ is by construction an isometry w.r.t. the Lip-norms $L_M$, $L_N$, since $|L_M(x)-L_N(\varphi(x))| \leq L_0(x,-\varphi(x)) = 0$. We want to show that $\varphi$ is an affine map. To this aim, let $x_1, x_2 \in X_M$ and let $y_1, y_2$ be the corresponding elements in $X_N$ for which $L_0 (x_i, -y_i) = 0$, $i = 1,2$. Then, for any $t \in [0,1]$, we have
\begin{eqnarray*}
& & L_0 (t x_1 + (1-t) x_2,  -(t y_1 + (1-t) y_2)) \\
& & = L_0 (t (x_1,  -y_1) + (1-t) (x_2,  -y_2)) \\
& & \leq t L_0 (x_1,  -y_1) + (1-t) L_0 (x_2,  -y_2) = 0,
\end{eqnarray*}
and thus
$$
\varphi(t x_1 + (1-t) x_2) = t y_1 + (1-t) y_2 = t \varphi(x_1) + (1-t) \varphi(x_2),
$$
showing that $\varphi$ is affine.

Now, since $\varphi$ is an affine bijective map from $X_M$ onto $X_N$, and $\varphi(0) = 0$ by Lm.~\ref{lm:unique1}, it extends to a bijective\footnote{The surjectivity simply follows by construction. As for the injectivity, let $x_1, x_2 \in \mathcal{M}_2(M)_{+}$, and assume that $\varphi(x_1 - x_2) = 0$; then, $L_0(x_1 - x_2, -\varphi(x_1 - x_2)) \leq L_0(x_1, -\varphi(x_1)) + L_0(x_2, -\varphi(x_2)) = 0$, and thus, as $L_0(x_1 - x_2, 0) = 0 = L_0(0, 0)$, by Lm.~\ref{lm:unique1} we get $x_1 - x_2 = 0$. Finally, let $x, y \in \cam_2(M)_{sa}$; then, since $\varphi(x + i y) = 0$  implies $\varphi(x) = \varphi(y) = 0$, by the first part the claim follows.} linear map from $\mathcal{M}_2(M)$ onto $\mathcal{M}_2(N)$ (we denote it still by $\varphi$), which is automatically positive and unital. In fact, we have
\begin{eqnarray*}
& & 0 \leq x_1 \leq x_2 \leq I_{\mathcal{M}_2(M)} \Rightarrow 0 \leq x_2 - x_1 \leq I_{\mathcal{M}_2(M)} \Rightarrow \\
& & 0 \leq \varphi (x_2 - x_1) \leq I_{\mathcal{M}_2(N)} \Rightarrow \varphi (x_1) \leq \varphi (x_2),
\end{eqnarray*}
and, analogously, $0 \leq y_1 \leq y_2 \leq I_{\mathcal{M}_2(N)}$ implies $\varphi^{-1} (y_1) \leq \varphi^{-1} (y_2)$. Thus, 
\begin{eqnarray*}
& & \varphi^{-1} (y) \leq I_{\mathcal{M}_2(M)} \quad \forall y \in X_N \, \Rightarrow \varphi^{-1} (I_{\mathcal{M}_2(N)}) \leq I_{\mathcal{M}_2(M)} \\
& & \varphi(x) \leq I_{\mathcal{M}_2(N)} \qquad \forall x \in X_M \Rightarrow \varphi (I_{\mathcal{M}_2(M)}) \leq I_{\mathcal{M}_2(N)},
\end{eqnarray*}
{\em i.e.}, $\varphi (I_{\mathcal{M}_2(M)}) = I_{\mathcal{M}_2(N)}$ and $\varphi^{-1} (I_{\mathcal{M}_2(N)}) = I_{\mathcal{M}_2(M)}$. 

Finally, since any $x \in \cam_2(M)_1$ admits a canonical decomposition $x = x_{1,+} - x_{1,-} + i (x_{2,+} - x_{2,-})$ with $x_{\pm,1}, x_{\pm,2} \in \cam_2(M)_{1,+}$, we get
$$
L_0(x, -\varphi(x)) \leq \sum_{i=1,2} \left( L_0(x_{i,+}, -\varphi(x_{i,+})) + L_0(x_{i,-}, -\varphi(x_{i,-})) \right) = 0,
$$
and the claim follows, as $x \in \cam_2(M)$ implies $\frac{x}{\|x\|} \in \cam_2(M)_1$ and clearly $\|x\| \, \varphi(\frac{x}{\|x\|}) = \varphi(x)$.

\vspace{0.5cm}

{\em Claim 2. Let $x \in \cam_2(M)$, $\displaystyle{ x = \left(
\begin{array}{cc}
a_{11} & a_{12} \\ 
a_{21} & a_{22}
\end{array}
\right) }$ with $a_{ij} \in M$, and set $\displaystyle{ \varphi(x) = \left(
\begin{array}{cc}
a'_{11} & a'_{12} \\ 
a'_{21} & a'_{22}
\end{array}
\right) }$ with $\displaystyle{ a'_{ij} = \varphi_{ij} \left( \left(
\begin{array}{cc}
a_{11} & a_{12} \\ 
a_{21} & a_{22}
\end{array}
\right) \right)} \in N$. Then,}
$$
L_0 \left( a_{ij}, -\varphi_{ij} \left( \left(
\begin{array}{cc}
a_{11} & a_{12} \\ 
a_{21} & a_{22}
\end{array}
\right) \right) \right) = 0, \quad i,j = 1,2.
$$

It is evident, as $\max_{ij} (L_0 (a_{ij}, -a'_{ij})) = L_0(x, -\varphi(x)) = 0$.

\vspace{0.5cm}

{\em Claim 3. The map $\varphi$ is of the form $\id_2 \otimes \phi$, with $\phi$ a bijective positive linear map between $M$ and $N$.}

Notice that $L_0(0,-x) = 0$, $x \in N$, implies $x = 0$ by the analog of Lm.~\ref{lm:unique1} with $M$, $N$ replacing $\cam_2(M)$, $\cam_2(N)$. Therefore introducing the matrix units $e_{ij}$, $i,j=1,2$, one sees, thanks to {\em Claim  2}, that $\varphi_{hk}(e_{ij}\otimes a_{ij}) = 0$ if $(h,k) \neq (i,j)$. Moreover, again by  {\em Claim  2}, $\varphi_{ij}(e_{ij}\otimes a)$, $a \in M$, is uniquely determined by
\[
L_0(a,-\varphi_{ij}(e_{ij}\otimes a)) = 0,
\]
and therefore $\phi(a) := \varphi_{ij}(e_{ij}\otimes a)$ is independent of $(i,j)$ and linear. Then one has, by linearity of $\varphi$,
\[
\varphi(x) = \sum_{i,j=1}^2 \varphi(e_{ij}\otimes a_{ij}) = \sum_{i,j=1}^2 e_{ij}\otimes \phi(a_{ij})=(\id_2\otimes\phi)(x),
\]
which easily implies bijectivity and positivity of $\phi$.

\vspace{0.5cm}

As $\phi$ is injective, unital and positive, it is an order-isomorphism onto its image; furthermore, being surjective, it is a $C^*$-homomorphism by Thm.~6.4 in \cite{Sto}. Reversing the roles of $M$ and $N$, we see that also $\phi^{-1}: N \to M$ is a unital, 2-positive bijective $C^*$-homomorphism. Therefore, by a result of Choi (see Corollary 3.2 in \cite{Choi}), $\phi$ is indeed a ${}^*$-isomorphism and there holds
\[
L_N(\phi(x)) = L_N(\varphi(e_{11}\otimes x)) = L_M(e_{11}\otimes x) = L_M(x), \quad x \in M.
\]

The proof is now complete.
\end{proof}

Let us observe that we might define the distance $\dist_{qGH^*}$ in the following equivalent way. Given two Lip-von Neumann algebras $(M, L_M)$, $(N, L_N)$, we consider all the pairs $(R, L_R)$ with $R$ a von Neumann algebra and $L_R$ a seminorm on $R$ such that there exist $2$-positive isometric embeddings 
\begin{eqnarray*}
& & \phi_M: M \to R, \quad L_R(\phi_M(\cdot)) = L_M(\cdot), \\
& & \phi_N: N \to R, \; \quad L_R(\phi_N(\cdot)) = L_N(\cdot).
\end{eqnarray*}
We set $\varphi_M = \id_2 \otimes \phi_M$ and $\varphi_N = \id_2 \otimes \phi_N$, and denote by $\mathcal{L}_{ie} \equiv \mathcal{L}_{ie} (M, N)$ the set of all such triples $(R, \phi_M, \phi_N)$. We then define
\begin{equation*}
{\rm dist}_{qGH^*}^{ie}(M, N) := \inf \{ {\rm dist}^{R}_H (\varphi_M(X_M), \varphi_N(X_N)) : (R, \phi_M, \phi_N)  \in \mathcal{L}_{ie} \},
\end{equation*}
where ``$ie$'' stands for ``isometric embedding''.

\begin{prop}\label{prop:equivdist}
{\em For any pair of Lip-von Neumann algebras $(M, L_M)$, $(N, L_N)$, we have:}
\begin{equation}
{\rm dist}_{qGH^*}(M, N) = {\rm dist}_{qGH^*}^{ie}(M, N).
\end{equation}
\end{prop}

\begin{proof}
Clearly, ${\rm dist}_{qGH^*}(M, N) \geq {\rm dist}_{qGH^*}^{ie}(M, N)$, since $R = M \oplus N$, with $\phi_M = \iota_M$, $\phi_N = \iota_N$ and $\iota_M, \iota_N$ the canonical embeddings, is just a particular choice, and, on the r.h.s., we take the infimum over all such choices. For the reverse inequality, let $(R, \phi_M, \phi_N) \in \mathcal{L}_{ie}$ be given. We then get a seminorm $L \in \cl(R, R)$ simply by setting $L(x,  y) := L_R(x + y)$. $L$ is clearly a seminorm that restricts to $L_R$ on each summand. We define $L_{M \oplus N}$ as the restriction of $L$ to $\phi_M(M) \oplus \phi_N(N)$. 
\\
Then, since $L_{M \oplus N} \in \cl(\varphi_M(M), \varphi_N(N))$ implies $L_{M \oplus N} \circ (\varphi_M \oplus \varphi_N) \in \cl(M, N)$, we have
\begin{equation*}\begin{split}
{\rm dist}_{qGH^*} (M, N) &\leq {\rm dist}^{L_{M \oplus N}}_H (\varphi_M(X_M), \varphi_N(X_N)) \\
&= {\rm dist}^{(R \oplus R, L)}_H (\varphi_M(X_M) \oplus \{ 0 \}, \{ 0 \} \oplus \varphi_N(X_N)) \\
&\leq {\rm dist}^{(R \oplus R, L)}_H (\varphi_M(X_M) \oplus \{ 0 \}, \{ 0 \} \oplus \varphi_M(X_M)) \\
& \quad + \; {\rm dist}^{(R \oplus R, L)}_H (\{ 0 \} \oplus \varphi_M(X_M), \{ 0 \} \oplus \varphi_N(X_N)) \\
&= {\rm dist}^{R}_H (\varphi_M(X_M), \varphi_N(X_N)),
\end{split}\end{equation*}
as ${\rm dist}^{(R \oplus R, L)}_H (\varphi_M(X_M) \oplus \{ 0 \}, \{ 0 \} \oplus \varphi_M(X_M)) = 0$. By taking the infimum over all the triples $(R, \phi_M, \phi_N)$, we have ${\rm dist}_{qGH^*}(M, N) \leq {\rm dist}_{qGH^*}^{ie}(M, N)$, and the thesis follows.
\end{proof}

\begin{thm}\label{thm:Lic}
{\em ${\rm dist}_{qGH^*}$ is a metric on the space of isomorphism classes of Lip-von Neumann algebras.}
\end{thm}

\begin{proof}
By Thm.~\ref{thm:dzero}, we already know that, if ${\rm dist}_{qGH^*} (M, N) = 0$, then $M$ and $N$ are isomorphic as LvNA's.

We show now the reverse implication. Let $\psi: M \to N$ be an isomorphism of LvNA's from  $(M, L_M)$ onto $(N, L_N)$. We set $R := N \oplus N$, $\phi_M := \psi \oplus 0$, $\phi_N := 0 \oplus \iota_N$, where $\iota_N$ is the identity map on $N$, and we define the following seminorm on $R$:
$$
L_R (y_1,y_2) := L_N(y_1 + y_2).
$$
Notice that $L_R(\phi_M(x)) = L_M(x)$ for any $x \in M$, and $L_R(\phi_N(y)) = L_N(y)$ for any $y \in N$. Then, by the previous Proposition, we have
$$
{\rm dist}_{qGH^*}(M, N) \leq {\rm dist}_{H}^R(\varphi_M(X_M), \varphi_N(X_N)).
$$
As ${\rm dist}_{H}^R(\varphi_M(X_M), \varphi_N(X_N)) = 0$, we get the claim.
\end{proof}

\begin{remark}
Let us notice that the distance ${\rm dist}_{qGH^*}$ does not appear to be complete, essentially because we do not have an estimate for the Lip-norm of products of elements, much like in the Rieffel's setting (see \cite{GI}).
\end{remark}

As mentioned briefly at the beginning of the present subsection, if one restricts to the von Neumman algebras acting on a separable Hilbert space $\ch$, the relation between the distance just introduced and the one inducing the Effros-Mar\'echal topology is analogous to the one between the Gromov-Hausdorff and the Hausdorff distance between compact subsets of a metric space. 

More in detail, the distance $\r(x, y) = \sum_{m, n} \frac{1}{2^{m + n}} | (\xi_m, (x - y) \xi_n)|$, $x, y \in \cb(\ch)_1$, with $\{ \xi_n \}_{n \in \bn}$ a dense subset in $\ch_1$, metrizes the $wo$-topology on bounded subsets of $\cb(\ch)$ (see, e.g., \cite{Tak}, Prop.~II.2.7), and therefore it can be used to define $\dist_{EM}$ as in Sec.~\ref{subsec:EMtop}. On the other hand, by translation-invariance (and positive homogeneity), $L_{\ch}(x) := \r(x, 0)$, $x \in \cb(\ch)$, is a dual-Lip-norm on $\cb(\ch)$, and the restriction $L_M := L_{\ch}|_{M}$, $M \subset \cb(\ch)$ a von Neumann algebra, is a dual-Lip-norm on $M$.

It is then clear that one can not expect that convergence in ${\rm dist}_{qGH^*}$ (with respect to these dual-Lip-norms) implies convergence in $\dist_{EM}$ as there can be distinct von Neumann subalgebras of $\cb(\ch)$ which are isomorphic as LvNA's. Conversely, if one defines
\[
\dist^+_{EM}(M_1,M_2) := \dist^\rho_H((M_1)_{1,+},(M_2)_{1,+})
\]
it is not difficult to see that if a sequence $\{M_n\}$ of von Neumman algebras on $\ch$ converges to $M$ w.r.t. $\dist^+_{EM}$ it converges also w.r.t. both $\dist_{EM}$ and $\dist_{qGH^*}$.

\subsection{A class of dual-Lip-norms}
In view of the subsequent application to the study of limits of vNA's associated to (bosonic) free fields, we now specialize to the case when we are given the same von Neumann algebra (as a subalgebra of some $\cb(\ch)$ with $\ch$ separable) endowed with  different dual-Lip-norms of a specific type.

\begin{lm}\label{lm:newdist}
{\em Let $L_1$ and $L_2$ be two dual-Lip-norms on the same von Neumann algebra $M$, and let $J \in \cl((M, L_1), (M, L_2))$. Then,}
$$
{\rm dist}_{qGH^*}((M, L_1), (M, L_2)) \leq \sup_{x \in M, \| x \| = 1} J(x,  -x).
$$
\end{lm}

\begin{proof}
By definition, we have
$$
{\rm dist}_{qGH^*}((M, L_1), (M, L_2)) \leq {\rm dist}_{H}^{J}((M, L_1), (M, L_2)),
$$
and, by definition of Hausdorff distance,
\begin{equation*}\begin{split}
{\rm dist}_{H}^{J}((M, L_1), (M, L_2)) &= \sup_{x \in X_M}\inf_{y \in X_M} J(x,-y) \leq \sup_{x \in X_M} J(x, -x) \\
&= \sup_{(x_{ij}) \in X_M} \max_{ij} J(x_{ij}, -x_{ij}) \\
&\leq \sup_{x \in M, \| x \| = 1} J(x, -x),
\end{split}\end{equation*}
where the last inequality follows from that fact that $X_M   \subseteq \cam_2(M)_1 \subseteq \cam_2(M_1)$.
\end{proof}

\begin{prop}\label{prop:newLN}
{\em Let $M$ be a von Neumann algebra acting on a separable Hilbert space $\ch$ and $\Omega \in \ch$ be a separating vector for $M$. Consider furthermore $T \in \cb(\ch)$ s.t. $\mathrm{Ker} \, T = \{ 0  \}$ and the mapping $M \ni x \mapsto  T x \Omega \in \ch$ is compact. Then, setting $L(x) := \| T x \Omega \|$, $L$ is a dual Lip-norm on $M$. }
\end{prop}

\begin{proof}
$L$ is clearly a seminorm; moreover, $L(x) = 0$ iff $\| T x \Omega \| = 0$ iff $T x \Omega = 0$ iff $x \Omega = 0$ iff $x = 0$, as $\Omega$ is separating for $M$ and $T$ is injective, showing that $L$ is indeed a norm.

Finally, by definition, the space $(M_1,L)$ is omeomorphic to the space $(\{T x \Omega, x\in M_1\},\|\ \|)$, which is compact by hypothesis. The result follows by Proposition \ref{DLSiff}(iv).
%
\end{proof}

\begin{lm}\label{lm:Jdis}
{\em Let $M$ be a von Neumann algebra acting on a separable Hilbert space $\ch$, and let $L_1$, $L_2$ be dual-Lip-norms on $M$ with $L_1 (\cdot) := \| T \cdot \Omega \|$ and $L_2 (\cdot) := \| S \cdot \Omega \|$, and assume that the operators $T$ and $S$ satisfy the same conditions as in the previous Proposition. Then,
\begin{equation}
J(x,  y) := \| T x \Omega +  S y \Omega \|
\end{equation}
belongs to $\cl((M, L_1), (M, L_2))$, and we have
$$
{\rm dist}_{qGH^*}((M, L_1), (M, L_2)) \leq \sup_{x \in M, \| x \| = 1}\| (T - S) x \Omega \|.
$$}
\end{lm}

\begin{proof}
$J$ is clearly a seminorm, and the restrictions of $J$ to the subalgebras $M \oplus \{0\}$ and $\{0\} \oplus M$ coincide with $L_1$ and $L_2$, respectively, implying $J \in \cl((M, L_1), (M, L_2))$. Finally, as $J(x,  -x) = \| (T - S) x \Omega \|$, Lemma \ref{lm:newdist} yields the last statement.
\end{proof}

\section{An application to free quantum fields}
As an application of the theory developed in the previous sections, we  are going to show that the local von Neumann algebras of the free quantum scalar field endowed with suitable dual-Lip-norms depend continuously on the field mass $m$ with respect to the dual quantum Gromov-Hausdorff distance.

For the convenience of the reader, and to fix notations, we recall here briefly the main definitions, see, \emph{e.g.},~\cite[Sec.\ X.7]{RS2} or \cite[Sec.\ 3.5]{Ara} for more details. For an open bounded set $O \subset \br^4$, we denote by $\ca_m(O)$ the local von Neumann algebra of the free quantum scalar field $\phi_m$ of mass $m \geq 0$, \emph{i.e.}, the von Neumann algebra generated by the Weyl operators $W_m(f) = e^{i\phi_m(f)}$, $f \in \cd_\br(O)$ (the real valued, $C^\infty(\br^4)$ functions with support in $O$), acting on the (separable) symmetric Fock space $\ch$ over $L^2(\br^3)$,
\[
\ch := \bc\oplus\bigoplus_{n=1}^{+\infty} L^2(\br^3)^{\otimes_Sn}.
\]

Moreover, we denote by $H_m$ the corresponding hamiltonian operator, \emph{i.e.}, the self-adjoint generator of the strongly continuous one-parameter unitary group obtained through second quantization from the unitary group $t \in \br \mapsto u_m(t)$ on $L^2(\br^3)$ defined by
\[
(u_m(t)\psi)(p) := e^{it\omega_m(p)}\psi(p), \qquad \psi \in L^2(\br^3),
\] 
where $\omega_m(p) := \sqrt{m^2+p^2}$, $p \in \br^3$. The vector $\Omega =(1,0,0,\dots)\in \ch$, called the vacuum vector, is the (up to a phase) unique unit vector invariant for $e^{itH_m}$ for all $t \in \br$, and it is separating for the local algebras $\ca_m(O)$, $m \geq 0$, by the Reeh-Schlieder theorem~\cite[Thm.\ 4.14]{Ara}. Moreover, it is well known that given $O \subset \br^4$ and $\beta > 0$, the map
\[
\Theta_m : \ca_m(O) \to \ch, \qquad \Theta_m(A) := e^{-\beta H_m}A\Omega,
\]
 is compact (and actually nuclear) for all $m \geq 0$~\cite{BP}.

It is also well known~\cite{EF} that for each $O$ and $m > 0$ there exists an isomorphism of von Neumann algebras $\tau_m : \ca_m(O) \to \ca_0(O)$ such that $\tau_m(W_m(f)) = W_0(f)$ for all $f \in \cd_\br(O)$. We also denote by $\tau_0 : \ca_0(O) \to \ca_0(O)$ the identity isomorphism.

The above facts suggest the definition of a natural family of Lip-norms $L_m$ on $\ca_m(O)$, $m \geq 0$.

\begin{prop}\label{prop:freelipnorm}
{\em The map $L_m$, $m \geq 0$, defined by
\begin{align*}
L_m(A) &:= \| e^{-\beta H_m}\tau_m(A)\Omega\|, \quad A \in \ca_m(O),
\end{align*}
is a dual-Lip-norm on $\ca_m(O)$.}
\end{prop}

\begin{proof}
Consider first the case $m=0$. Then $L_0$ is a dual-Lip-norm thanks to the separating property of $\Omega$ and to the compactness of $\Theta_{0}$, as follows from Prop.~\ref{prop:newLN}. 

The analogous statement for $L_m$, $m > 0$, is obtained by observing that, for all $\Psi \in \ch$,
\[
(e^{-\beta H_m} e^{\beta H_0}\Psi)_n(p_1,\dots,p_n) = e^{-\beta\sum_{j=1}^n [\omega_m(p_j)-\omega_0(p_j)]}\Psi_n(p_1,\dots,p_n),
\]
which, together with $\omega_m(p) \geq \omega_0(p)$, $p \in \br^3$, implies $\|e^{-\beta H_m} e^{\beta H_0}\| \leq 1$. Therefore
\[
L_m\circ\tau_m^{-1}(A) = \| e^{-\beta H_m}A \Omega\| \leq \|\Theta_0(A)\|,  \qquad A \in \ca_0(O),
\] 
which implies that $A \in \ca_0(A) \mapsto e^{-\beta H_m} A \Omega$ is also compact, and then $L_m\circ\tau_m^{-1}$ is a dual-Lip-norm on $\ca_0(O)$, again by Prop.~\ref{prop:newLN}. Being $\tau_m$ isometric and bi-$w^*$-continuous, we conclude that $L_m$ is a dual-Lip-norm on $\ca_m(O)$.
\end{proof}

In order to prove convergence of $\ca_m(O)$ to $\ca_0(O)$ with respect to the above defined dual-Lip-norms, we note the following general fact.

\begin{lm}\label{lm:uniformcompact}
{\em Let $X$, $Y$ be Banach spaces, $\Theta : X \to Y$ a compact linear map, and $T(s)$, $s > 0$, a uniformly bounded family of operators on $Y$ such that $T(s) \to 0$ strongly as $s \to 0$. Then
\[
\lim_{s\to 0} \sup_{x \in X_1} \|T(s)\Theta(x)\| = 0.
\]}
\end{lm}

\begin{proof}
It is not restrictive to assume that $\| T(s)\| \leq 1$ for all $s > 0$. Since $\Theta$ is compact, it maps the unit ball $X_1$ to a totally bounded subset of $Y$, namely for each given $\varepsilon > 0$ there exist finitely many elements $y_1,\dots, y_n \in \Theta(X_1)$ such that
\[
\Theta(X_1) \subset \bigcup_{j=1}^n B_{\varepsilon/2}(y_j),
\]
and for each $j=1,\dots,n$ we can find $\delta_j > 0$ such that $\|T(s)y_j\| < \varepsilon/2$ if $s < \delta_j$. Therefore if $s < \delta := \min_{j}\delta_j$ and $x \in X_1$, given $y_j$ such that $\|\Theta(x)-y_j\| < \varepsilon/2$, one has
\[
\|T(s)\Theta(x)\| \leq \|\Theta(x)-y_j\| + \|T(s)y_j\| < \varepsilon,
\]
and hence the statement.
\end{proof}

The main result of this Section is therefore the following.

\begin{thm}
{\em 
The family of von Neumann algebras $\ca_m(O)$, $m \geq 0$, is continuous with respect to the dual quantum Gromov-Hausdorff distance defined by the dual-Lip-norms $L_m$  defined in Prop.~\ref{prop:freelipnorm}, namely
\[
\lim_{m'\to m }{\rm dist}_{qGH^*}\big(\ca_{m'}(O),\ca_m(O))= 0, \qquad m \geq 0.
\]}
\end{thm}

\begin{proof}
Since $L_m\circ \tau_m^{-1} = \| e^{-\beta H_m}(\cdot)\Omega\|$ is a dual-Lip-norm on $\ca_0(O)$, as observed in the proof of~Prop.~\ref{prop:freelipnorm}, the Lip-von Neumann algebras $(\ca_m(O),L_m)$ and $(\ca_0(O),L_m\circ\tau_m^{-1})$ are isomorphic, and then, by Thm.~\ref{thm:Lic}
\begin{multline*}
{\rm dist}_{qGH^*}\big(\ca_{m'}(O),\ca_m(O)\big)\\
={\rm dist}_{qGH^*}\big((\ca_0(O),L_{m'}\circ\tau_{m'}^{-1}),(\ca_0(O),L_{m}\circ\tau_{m}^{-1})\big).
\end{multline*}
Moreover, setting
\[
J(A, B) := \|e^{-\beta H_{m'}}A\Omega + e^{-\beta H_m} A \Omega\|, \qquad (A, B) \in \ca_0(O)\oplus\ca_0(O),
\]
by Lm.~\ref{lm:Jdis}, $J$ is a seminorm on $\ca_0(O)\oplus\ca_0(O)$ which restricts to $L_{m'}\circ \tau_{m'}^{-1}$ and $L_m\circ \tau_m^{-1}$, and therefore, again by Lemma \ref{lm:Jdis},
\begin{multline*}
 {\rm dist}_{qGH^*}\big(\ca_{m'}(O),\ca_m(O)) \leq \sup_{A \in \ca_0(O)_1}\|(e^{-\beta H_{m'}}-e^{-\beta H_{m}}) A \Omega\|\\
 =\sup_{A \in \ca_0(O)_1}\|(e^{-\beta H_{m'}}e^{\beta H_0}-e^{-\beta H_{m}}e^{\beta H_0}) \Theta_0(A)\|.
\end{multline*}
Thanks to Lm.~\ref{lm:uniformcompact}, the statement is then achieved if we can show that $e^{-\beta H_{m'}}e^{\beta H_0}-e^{-\beta H_{m}}e^{\beta H_0} \to 0$ strongly as $m' \to m$. In order to do that, take $\Psi_n \in \ch$, a vector whose only non-vanishing component lies in $L^2(\br^3)^{\otimes_Sn}$. Then
\begin{multline*}
\|(e^{-\beta H_{m'}}e^{\beta H_0}-e^{-\beta H_{m}}e^{\beta H_0})\Psi_n\|^2 \\
= \int_{\br^{3n}}dp_1\dots dp_n \big|\big(e^{-\beta\sum_{j=1}^n [\omega_{m'}(p_j)-\omega_0(p_j)]}\\
-e^{-\beta\sum_{j=1}^n [\omega_{m}(p_j)-\omega_0(p_j)]}\big)\Psi_n(p_1,\dots,p_n)\big|^2,
\end{multline*}
and an application of the dominated convergence theorem shows that $\|(e^{-\beta H_{m'}}e^{\beta H_0}-e^{-\beta H_{m}}e^{\beta H_0})\Psi_n\| \to 0$. The required strong convergence is then obtained by observing that vectors of the form $\Psi_n$ span a dense subspace of $\ch$, and thanks to the uniform boundedness, in $m' \geq 0$, of $\|e^{-\beta H_{m'}}e^{\beta H_0}-e^{-\beta H_{m}}e^{\beta H_0}\|$.
\end{proof}

It is an interesting open question wether an analogous result holds with respect to other natural Lip-norms, such as, e.g., $A \in \ca_m(O) \mapsto \|e^{-\beta H_m}A \Omega\|$, $m \geq 0$.


\medskip
\emph{Acknowledgements.} One of the authors (D.~G.) wishes to thank R.~Verch and H.~Bostelmann for enlightening discussions at a very early stage of this work.

We  thank the anonymous referee for a very careful reading of the manuscript, and the suggestion of many  changes which improved the quality of the paper.

This work was supported by the following institutions:
the ERC for the Advanced Grant 227458 OACFT ``Operator Algebras and Conformal Field Theory'', the MIUR PRIN ``Operator Algebras, Noncommutative Geometry and Applications'', the INdAM-CNRS GREFI GENCO, and the INdAM GNAMPA.

\end{document}